\documentclass{amsart}
\usepackage{amscd,amssymb}

\newcounter{noindnum}[subsection]
\renewcommand{\thenoindnum}{\alph{noindnum}}
\newcommand{\noindstep}{\refstepcounter{noindnum}{{\rm(}\thenoindnum}\/{\rm)} }
\newcommand{\stepzero}{\setcounter{noindnum}{0}}
\renewcommand{\phi}{\varphi}
\renewcommand{\epsilon}{\varepsilon}
\renewcommand{\emptyset}{\varnothing}

\newcommand{\bE}{\mathbf E}
\newcommand{\bG}{\mathbf G}
\newcommand{\bH}{\mathbf H}
\newcommand{\bP}{\mathbf P}
\newcommand{\bT}{\mathbf T}
\newcommand{\bU}{\mathbf U}
\newcommand{\bu}{\mathbf u}

\newcommand{\cA}{\mathcal A}
\newcommand{\cB}{\mathcal B}
\newcommand{\cE}{\mathcal E}
\newcommand{\cF}{\mathcal F}
\newcommand{\cG}{\mathcal G}
\newcommand{\cH}{\mathcal H}
\newcommand{\cO}{\mathcal O}
\newcommand{\cP}{\mathcal P}

\renewcommand{\P}{\mathbb P}
\newcommand{\A}{\mathbb A}
\newcommand{\Z}{\mathbb{Z}}

\newcommand{\fm}{\mathfrak m}

\DeclareMathOperator{\Proj}{Proj}
\DeclareMathOperator{\Pic}{Pic}
\DeclareMathOperator{\Rad}{Rad}
\DeclareMathOperator{\Gl}{Gl}
\DeclareMathOperator{\Nneib}{Neib}
\DeclareMathOperator{\wN}{\widetilde{Neib}}

\DeclareMathOperator{\Id}{Id}
\DeclareMathOperator{\Aut}{Aut}
\DeclareMathOperator{\Iso}{Iso}
\DeclareMathOperator{\spec}{Spec}
\DeclareMathOperator{\Hom}{Hom}
\DeclareMathOperator{\Ker}{Ker}

\theoremstyle{plain}
\newtheorem{theorem}{Theorem}
\newtheorem*{theorem*}{Theorem}
\newtheorem{conjecture}{Conjecture}
\newtheorem{proposition}{Proposition}[section]
\newtheorem{lemma}[proposition]{Lemma}
\newtheorem{corollary}{Corollary}
\newtheorem*{corollary*}{Corollary}

\theoremstyle{definition}
\newtheorem{definition}[proposition]{Definition}
\newtheorem{construction}[proposition]{Construction}

\theoremstyle{remark}
\newtheorem{remark}[proposition]{Remark}
\newtheorem*{remarks*}{Remarks}
\newtheorem*{remark*}{Remark}

\begin{document}

\title[A proof of the Grothendieck--Serre conjecture]{A proof of the Grothendieck-Serre conjecture on principal bundles over regular local rings containing infinite fields}


\begin{abstract}
Let $R$ be a regular local ring containing an infinite field. Let $\bG$ be a reductive group scheme over~$R$. We prove that a principal $\bG$-bundle over~$R$ is trivial if it is trivial over the fraction field of $R$. In other words, if $K$ is the fraction field of $R$, then the map of non-abelian cohomology pointed sets
\[
  H^1_{\text{\'et}}(R,\bG)\to H^1_{\text{\'et}}(K,\bG)
\]
induced by the inclusion of $R$ into $K$ has a trivial kernel.
\end{abstract}

\author{Roman Fedorov}
\email{rmfedorov@gmail.com}
\address{Mathematics Department, 138 Cardwell Hall, Kansas State University, Manhattan, KS 66506, USA}

\author{Ivan Panin}
\email{paniniv@gmail.com}
\address{Steklov Institute of Mathematics at St.-Petersburg, Fontanka 27, St.-Petersburg 191023, Russia}

\maketitle

\section{Introduction}
Assume that $U$ is a regular scheme. Let $\bG$ be a reductive $U$-group scheme, that is, $\bG$ is affine and smooth as a $U$-scheme and, moreover, the geometric fibers of $\bG$ are connected reductive algebraic groups (see~\cite[Exp.~XIX, Def.~2.7]{SGA3-3}).

Recall that a $U$-scheme $\cG$ with an action of $\bG$ is called \emph{a principal $\bG$-bundle over~$U$}, if $\cG$ is faithfully flat and quasi-compact over $U$ and the action is simply transitive, that is, the natural morphism $\bG\times_U\cG\to\cG\times_U\cG$ is an isomorphism  (see~\cite[Sect.~6]{FGA}). It is well known that such a bundle is trivial locally in the \'etale topology but in general not in the Zariski topology. Grothendieck and Serre conjectured that if $\cG$ is generically trivial, then it is locally trivial in the Zariski topology (see~\cite[Remarque, p.31]{Se}, \cite[Remarque 3, p.26-27]{Gr1}, and~\cite[Remarque~1.11.a]{Gr2}). More precisely, the following conjecture is widely attributed to them.
\begin{conjecture}\label{conj}
Let $R$ be a regular local ring, let $K$ be its field of fractions. Let~$\bG$ be a reductive group scheme over $U:=\spec R$, let $\cG$ be a principal $\bG$-bundle. If~$\cG$ is trivial over $\spec K$, then it is trivial. Equivalently, the map of non-abelian cohomology pointed sets
\[
  H^1_{\text{\'et}}(R,\bG)\to H^1_{\text{\'et}}(K,\bG)
\]
induced by the inclusion of $R$ into $K$ has a trivial kernel.
\end{conjecture}
The main result of this paper is the following theorem.
\begin{theorem*}
The above conjecture holds if $R$ is a regular local ring containing an infinite field.
\end{theorem*}
The theorem has the following corollary.
\begin{corollary*}
Notation as in the conjecture, two principal $\bG$-bundles over $U$ that become isomorphic upon restriction to $\spec K$ are isomorphic.
\end{corollary*}
This result is new even for constant group schemes (that is, for group schemes coming from the ground field).

\subsection{History of the topic}
In his 1958 paper Jean--Pierre Serre asked whether a principal bundle is Zariski locally trivial, once it has a rational section  (see~\cite[Remarque, p.31]{Se}). In his setup the group is any algebraic group over an algebraically closed field. He gave an affirmative answer to the question when the group is $PGL(n)$ (see~\cite[Prop.~18]{Se}) and when the group is an abelian variety (see~\cite[Lemme~4]{Se}). In the same year, Alexander Grothendieck asked a similar question (see~\cite[Remarque 3, p.26-27]{Gr1}).

A few years later, Grothendieck conjectured that the statement is true for any semi-simple group scheme over any regular scheme (see~\cite[Remarque~1.11.a]{Gr2}). Now by the Grothendieck--Serre conjecture we mean Conjecture~\ref{conj} though this may be slightly inaccurate from historical perspective. Many results corroborating the conjecture are known.

\smallskip $\bullet$ For some simple group schemes of classical series the conjecture is solved in works of the second author, A.~Suslin, M.~Ojanguren, and K.~Zainoulline; see~\cite{Oj1}, \cite{Oj2}, \cite{PS}, \cite{OP2}, \cite{Z}, \cite{OPZ}.

\smallskip $\bullet$ The case of an arbitrary reductive group scheme over a discrete valuation ring or over a Henselian ring is completely solved by Y.~Nisnevich in~\cite{Ni1}. He also proved the conjecture for two-dimensional local rings in the case when $\bG$ is quasi-split in~\cite{Ni2}.

\smallskip $\bullet$ The case where $\bG$ is an arbitrary torus over a regular local ring was settled by J.-L.~Colliot-Th\'{e}l\`{e}ne and J.-J.~Sansuc in~\cite{C-T-S}.

\smallskip $\bullet$ The case where the group scheme $\bG$ comes from an infinite ground field is completely solved by J.-L.~Colliot-Th\'el\`ene, M.~Ojanguren, and M.~S.~Raghunathan in~\cite{C-TO} and \cite{R1,R2}; O.~Gabber announced a proof for group schemes coming from arbitrary ground fields.

\smallskip $\bullet$ Under an isotropy condition on $\bG$ the conjecture is proved in a series of preprints~\cite{PSV} and~\cite{Pa2}.

\smallskip $\bullet$ The case of strongly inner simple adjoint group schemes of types $E_6$ and $E_7$ is done by the second author, V.~Petrov, and A.~Stavrova in~\cite{PPS}. No isotropy condition is imposed there.

\smallskip $\bullet$ The case when $\bG$ is of type $F_4$ with trivial $g_3$-invariant and the field is of characteristic zero is settled by V.~Chernousov in~\cite{Chernous}; the case when $\bG$ is of type $F_4$ with trivial $f_3$-invariant and the field is infinite and perfect is settled by V.~Petrov and A.~Stavrova in~\cite{PetrovStavrova}.

In the case of anisotropic group schemes the conjecture remained wide open in many cases, in particular, for group schemes of types $D_n$, $F_4$, and $E_8$. We will present a uniform proof.

After the first version of the current preprint was posted, the conjecture was solved for regular local rings containing finite fields by the second author~\cite{PaninFiniteFieldsLAG}. Thus, the conjecture holds for regular local rings containing a field. In the mixed characteristic case the conjecture was solved by the first author provided that~$\bG$ is split and some strong conditions on the local ring are satisfied~\cite{FedorovMixedChar}.

\subsection{Overview of the proof}
Very roughly, the idea of the proof is to relate the problem of triviality of the original principal bundle to the triviality of a~principal bundle over the affine line over $U$ (see Theorem~\ref{th:psv}) and then to triviality of a principal bundle over the projective line over $U$ (see Theorem~\ref{MainThm2}). The first reduction is based on the geometric part of the paper~\cite{PSV} by the second author with A.~Stavrova and N.~Vavilov. We also use results of the second author~\cite{Pa2} to reduce our problem to the case when $\bG$ is simple and simply-connected (at a price of replacing a local ring by semi-local). Also, by a result of Popescu~\cite{P,Sw,SpivakovskyPopescu} we may assume that $U$ is of geometric origin.

The proof of Theorem~\ref{MainThm2} is inspired by the theory of affine Grassmannians. We do not use the affine Grassmannians explicitly in this paper, however, the interested reader is invited to look at~\cite{FedorovExotic}, where an alternative proof of our Theorem~\ref{MainThm2} is sketched.

\subsection{Acknowledgments}
The authors collaborated during Summer schools ``Contemporary Mathematics'', organized by Moscow Center for Continuous Mathematical Education. They would like to thank the organizers of the school for the perfect working atmosphere. The authors benefited from talks with A.~Braverman. They also thank J.-L.~Colliot-Th{\'e}l{\`e}ne and P.~Gille for their interest to the topic.

A part of the work was done, while the first author was a member of Max Planck Institute for Mathematics in Bonn. The first author also benefited from talks with D.~Arinkin.
The first author was partially supported by the NSF grant DMS-1406532. The second author is very grateful to A.~Stavrova and N.~Vavilov for a stimulating interest to the topic. He also thanks the RFBR-grant 13--01--00429--a for the support. Theorem~\ref{th:psv} is proved with the support of the Russian Science Foundation (grant no.~14--11--00456).

The authors would like to thank the anonymous referees for valuable suggestions. The final publication is available at Springer via http://dx.doi.org/10.1007/s10240-\\ 015-0075-z.

\section{Main results}\label{Introduction}
The theorem from the introduction follows from a slightly more general result.

\begin{theorem}\label{MainThm1}
Let $R$ be a regular semi-local domain containing an infinite field, and let $K$ be its field of fractions. If $\bG$ is
a reductive group scheme over $R$, then the map
\[
  H^1_{\text{\'et}}(R,\bG)\to H^1_{\text{\'et}}(K,\bG)
\]
\noindent
induced by the inclusion of $R$ into $K$ has a trivial kernel. In other words, under the above assumptions on $R$ and $\bG$, each principal $\bG$-bundle over $R$ having a $K$-rational point is trivial.
\end{theorem}
Theorem~\ref{MainThm1} has the following corollary.
\begin{corollary}\label{corollary}
Under the same hypothesis as in Theorem~\ref{MainThm1}, the map
\[
  H^1_{\text{\'et}}(R,\bG)\to H^1_{\text{\'et}}(K,\bG)
\]
\noindent
induced by the inclusion of $R$ into $K$ is injective. Equivalently, two principal $\bG$-bundles over $R$ that become isomorphic upon restriction to $K$ are isomorphic.
\end{corollary}
\begin{proof}
Let $\cG_1$ and $\cG_2$ be two principal $\bG$-bundles over $U:=\spec R$. Assume that $\cG_1$ and $\cG_2$ are isomorphic over $\spec K$. Recall that the functor sending a $U$-scheme $T$ to the set of isomorphisms of principal $\bG$-bundles $\cG_1\times_UT\to\cG_2\times_UT$ is represented by an affine $U$-scheme $\Iso(\cG_1,\cG_2)$. Consider also the scheme $\Aut\cG_2:=\Iso(\cG_2,\cG_2)$ of $\bG$-bundle automorphisms of $\cG_2$. It is a reductive group scheme because it is \'etale locally over $R$ isomorphic to $\bG$.

It is easy to see that $\Iso(\cG_1,\cG_2)$ is a principal $\Aut\cG_2$-bundle. By Theorem~\ref{MainThm1} it is trivial, and we see that $\cG_1\cong\cG_2$.
\end{proof}

While Theorem~\ref{MainThm1} was previously known for reductive group schemes $\bG$ coming from the ground field (see~\cite{C-TO,R1,R2}), in certain cases the corollary is a new result even for such group schemes. For example, it was not known for split group schemes $\bG$ of type $E_8$. Also, the corollary was not known for $\mathop{\mathrm{Spin}}(A,\sigma)$, where $A$ is a skew-field over a field $k$ ($\mathop{\mathrm{char}} k\ne2$) and $\sigma$ is an involution of orthogonal type on $A$.

For a scheme $U$ we denote by $\A^1_U$ the affine line over $U$ and by $\P^1_U$ the projective line over $U$. If $T$ is a $U$-scheme, we will use the term ``principal $\bG$-bundle over $T$'' to mean a principal $\bG\times_UT$-bundle over $T$.

In Section~\ref{sect:redtopsv} we deduce Theorem~\ref{MainThm1} from the following result of independent interest (cf.~\cite[Thm.~1.3]{PSV}).

\begin{theorem}\label{th:psv}
Let $R$ be the semi-local ring of finitely many closed points on an irreducible smooth affine variety over an infinite field $k$ and set $U=\spec R$. Let $\bG$ be a~simple, simply-connected group scheme over $U$ {\rm(}see~\cite[Exp.~XXIV, Sect.~5.3]{SGA3-3} for the definition{\rm)}. Let $\cE_t$ be a principal $\bG$-bundle over the affine line $\A^1_U=\spec R[t]$, and let $h(t)\in R[t]$ be a monic polynomial. Denote by $(\A^1_U)_h$ the open subscheme in $\A^1_U$ given by $h(t)\ne0$ and assume that the restriction of $\cE_t$ to $(\A^1_U)_h$ is a trivial principal $\bG$-bundle. Then for each section $s:U\to\A^1_U$ of the projection $\A^1_U\to U$ the $\bG$-bundle $s^*\cE_t$ over $U$ is trivial.
\end{theorem}

The derivation of Theorem~\ref{MainThm1} from Theorem~\ref{th:psv} is based on results of the second author, A.~Stavrova, and N.~Vavilov, namely, on~\cite{Pa2} and~\cite[Thm.~1.2]{PSV}.

Let $Y$ be a semi-local scheme. We will call a simple $Y$-group scheme isotropic if its restriction to each connected component of $Y$ contains a proper parabolic subgroup scheme. (Note that by~\cite[Exp.~XXVI, Cor.~6.14]{SGA3-3}
this is equivalent to the usual definition, that is, to the requirement that the group scheme contains a torus isomorphic to~$\mathbb G_{m,Y}$.) Theorem~\ref{th:psv} is, in turn, derived from the following statement.

\begin{theorem}\label{MainThm2}
Let $R$ be the semi-local ring of finitely many closed points on an irreducible smooth affine variety over an infinite field $k$ and set $U=\spec R$. Let~$\bG$ be a simple, simply-connected group scheme over $U$.

Let $Z\subset\P^1_U$ be a closed subscheme finite over $U$. Let $Y\subset\P^1_U$ be a closed subscheme \'etale over $U$. Assume that $Y\cap Z=\emptyset$, and $\bG_Y:=\bG\times_UY$ is isotropic.
Assume also that for every closed point $u\in U$ such that the algebraic group $\bG_u:=\bG|_u$ is isotropic, there is a $k(u)$-rational point in $Y_u:=\P^1_u\cap Y$. {\rm(}Here $k(u)$ is the residue field of $u$.{\rm)}

Let $\cG$ be a~principal $\bG$-bundle over $\P^1_U$ such that its restriction to $\P^1_U-Z$ is trivial. Then the restriction of $\cG$ to $\P^1_U-Y$ is also trivial.
\end{theorem}

The proof of this result was inspired by the theory of affine Grassmannians (see~\cite{FedorovExotic} for a proof using affine Grassmannians explicitly).

\begin{remarks*}
1. Assume that for every closed point $u\in U$ the algebraic group $\bG_u$ is anisotropic. Then we can take $Y=\emptyset$.

2. It is not necessary to assume that $Y\cap Z=\emptyset$. Indeed, let $Y$ satisfy the conditions of the theorem except that it may intersect $Z$. Since $U$ is semi-local, there is a projective transformation $\theta:\P^1_U\to\P^1_U$ such that $\theta(Y)\cap Y=\theta(Y)\cap Z=\emptyset$. By the above theorem the restriction of $\cG$ to $\P^1_U-\theta(Y)$ is trivial. Now we can apply the theorem again with $Z=\theta(Y)$ to show that the restriction of $\cG$ to $\P^1_U-Y$ is trivial.

3. In the situation of Theorem~\ref{MainThm2}, let $\bG$ be isotropic. Then it follows from the theorem that one can take $Y=\{\infty\}\times U\subset\P^1_U$, that is, the restriction of $\cG$ to $\A^1_U$ is trivial. In fact, this is a partial case of~\cite[Thm.~1.3]{PSV}. On the other hand, if~$\bG$ is anisotropic, this restriction is not in general trivial. For an example see~\cite{FedorovExotic}.
\end{remarks*}

\subsection{Organization of the paper}
In Section~\ref{sect:redtopsv}, we reduce Theorem~\ref{MainThm1} to Theorem~\ref{th:psv}. This reduction is based on~\cite{Pa2}, \cite[Thm.~1.2]{PSV}, and a theorem of D.~Popescu~\cite{P,Sw,SpivakovskyPopescu}. In Section~\ref{sect:reducing}, we reduce Theorem~\ref{th:psv} to Theorem~\ref{MainThm2}.

In Section~\ref{sect:proof2} we prove Theorem~\ref{MainThm2}. The main idea is to modify the principal bundle $\cG$ in a neighborhood of $Y$ so that $\cG$ becomes trivial. We use the technique of Henselization. One can give an essentially equivalent proof based on formal loops, see~\cite[Sect.~6.2]{FedorovExotic}.

In Section~\ref{sect:application} we give an application of Theorem~\ref{MainThm1}.

\section{Reducing Theorem~\ref{MainThm1} to Theorem~\ref{th:psv}}\label{sect:redtopsv}
In what follows ``$\bG$-bundle'' always means ``principal $\bG$-bundle''.  Now we assume that Theorem~\ref{th:psv} holds. We start with the following particular case of Theorem~\ref{MainThm1}.

\begin{proposition}\label{pr:geometric}
Let $R$ be the semi-local ring of finitely many closed points on an irreducible smooth affine variety over an infinite field $k$ and set $U=\spec R$. Let $\bG$ be a simple, simply-connected group scheme over $U$. Let $\cE$ be a principal $\bG$-bundle over $U$, trivial at the generic point of $U$. Then $\cE$ is trivial.
\end{proposition}

\begin{proof}
Under the hypothesis of the proposition, a particular case of~\cite[Thm.~1.2]{PSV} reads as follows: there exist\\
\stepzero\noindstep a principal $\bG$-bundle $\cE_t$ over $\A^1_U$;\\
\noindstep a monic polynomial $h(t)\in R[t]$.\\
Moreover, these data satisfy the following conditions:\\
(1) the restriction of $\cE_t$ to $(\A^1_U)_h$ is a trivial principal $\bG$-bundle;\\
(2) there is a section $s:U\to\A^1_U$ such that $s^*\cE_t=\cE$.

Now it follows from Theorem~\ref{th:psv} that $\cE$ is trivial.
\end{proof}

\begin{proposition}\label{pr:reductivegeometric}
Let $R$ be the semi-local ring of finitely many closed points on an irreducible smooth affine variety over an infinite field $k$ and set $U=\spec R$. Let $\bG$ be a reductive group scheme over $U$.
Let $\cE$ be a principal $\bG$-bundle over $U$ trivial at the generic point of $U$. Then $\cE$ is trivial.
\end{proposition}

\begin{proof}
The following is proved in~\cite{Pa2}:
\begin{itemize}
\item Denote by $\bG_{der}$ the derived group scheme of $\bG$. If the Grothendieck--Serre conjecture holds for any inner form of $\bG_{der}$, then it holds for $\bG$. (Recall that an inner forms of a group scheme $\bH$ is a group scheme isomorphic to $\Aut(\cH)$, where $\cH$ is an $\bH$-bundle.)
\item If the Grothendieck--Serre conjecture holds for any inner form of the simply-connected cover of a semi-simple $U$-group scheme $\bH$, then it holds for $\bH$.
\end{itemize}

Thus, we may assume that $\bG$ is semi-simple and simply-connected. By~\cite[Exp.~XXIV, Prop.~5.10]{SGA3-3} (which is valid for simply-connected group schemes as well, see the beginning of~\cite[Exp.~XXIV, Sect.~5]{SGA3-3}) there is a sequence $U_1,\ldots,U_r$ of finite \'etale $U$-schemes, and for each $i=1,\ldots,r$ a simple simply-connected $U_i$-group scheme $\bG_i$ such that
\[
    \bG\cong\prod_{i=1}^r\mathrm{R}_{U_i/U}(\bG_i),
\]
where $\mathrm{R}_{U_i/U}$ is the Weil restriction functor. Now the Faddeev--Shapiro Lemma (see~\cite[Exp. XXIV, Prop.~8.4]{SGA3-3}) shows that the Grothendieck--Serre conjecture for $\bG$ holds, if for each $i$ the conjecture holds for $\bG_i$. For more details, see~\cite[Thm.~11.1]{PSV}. Thus, we may assume that $\bG$ is simple and simply-connected. Now the proposition is reduced to Proposition~\ref{pr:geometric}.
\end{proof}
\begin{remark}
Even if we start with a local scheme $U$, the schemes $U_i$ are only semi-local in general. This is why we have to work with semi-local schemes from the beginning.
\end{remark}

\begin{proof}[Proof of Theorem~\ref{MainThm1} assuming Theorem~\ref{th:psv}]
Let us prove a general statement first. Let $k'$ be an infinite field, $X$ be a $k'$-smooth irreducible affine variety, $\bH$ be a reductive group scheme over $X$. Denote by $k'[X]$ the ring of regular functions on $X$ and by $k'(X)$ the field of rational functions on $X$. Let $\cH$ be a principal $\bH$-bundle over $X$ trivial over $k'(X)$. Let $\mathfrak p_1,\dots,\mathfrak p_n$ be prime ideals in $k'[X]$, and let $\cO_{\mathfrak p_1,\dots,\mathfrak p_n}$ be the corresponding semi-local ring.
\begin{lemma}\label{lm:primemax}
The principal $\bH$-bundle $\cH$ is trivial over $\cO_{\mathfrak p_1,\dots,\mathfrak p_n}$.
\end{lemma}
\begin{proof}
For each $i=1,2,\ldots,n$ choose a maximal ideal $\mathfrak m_i\subset k'[X]$ containing $\mathfrak p_i$. One has inclusions of $k'$-algebras
\[
\cO_{\mathfrak m_1,\dots,\mathfrak m_n}\subset\cO_{\mathfrak p_1,\dots,\mathfrak p_n}\subset k'(X).
\]
By Proposition~\ref{pr:reductivegeometric} the principal $\bH$-bundle $\cH$ is trivial over $\cO_{\mathfrak m_1,\dots,\mathfrak m_n}$. Thus it is trivial over $\cO_{\mathfrak p_1,\dots,\mathfrak p_n}$.
\end{proof}

Let us return to our situation. Let $\fm_1,\ldots,\fm_n$ be all the maximal ideals of $R$. Let $\cE$ be a $\bG$-bundle over $R$ trivial over the fraction field of $R$. Clearly, there is a non-zero $f\in R$ such that $\cE$ is trivial over $R_f$. Let $k'$ be the algebraic closure of the prime field of $R$ in $k$. Note that $k'$ is perfect. It follows from Popescu's theorem (\cite{P,Sw,SpivakovskyPopescu}) that~$R$ is a filtered inductive limit of smooth $k'$-algebras $R_\alpha$. Modifying the inductive system $R_\alpha$ if necessary, we can assume that each $R_\alpha$ is integral.
There exist an index $\alpha$, a reductive group scheme $\bG_{\alpha}$ over $R_{\alpha}$, a principal $\bG_{\alpha}$-bundle $\cE_{\alpha}$ over $R_{\alpha}$, and an element $f_{\alpha }\in R_{\alpha}$ such that $\bG=\bG_{\alpha}\times_{\spec R_{\alpha}}\spec R$, $\cE$ is isomorphic to $\cE_{\alpha}\times_{\spec R_{\alpha}}\spec R$ as principal $\bG$-bundle, $f$ is the image of $f_{\alpha}$ under the homomorphism $\phi_{\alpha}: R_{\alpha}\to R$, and $\cE_{\alpha}$ is trivial over $(R_{\alpha})_{f_{\alpha}}$.

If the field $k'$ is infinite, then for each maximal ideal $\mathfrak m_i$ in~$R$ ($i=1,\dots, n$) set $\mathfrak p_i=\phi_{\alpha}^{-1}(\mathfrak m_i)$. The homomorphism $\phi_\alpha$ induces a homomorphism of semi-local rings $(R_{\alpha})_{\mathfrak p_1,\dots,\mathfrak p_n}\to R$. By Lemma~\ref{lm:primemax} the principal $\bG_{\alpha}$-bundle $\cE_{\alpha}$ is trivial over
$(R_{\alpha})_{\mathfrak p_1,\dots,\mathfrak p_n}$. Whence the $\bG$-bundle $\cE$ is trivial over $R$.

If the field $k'$ is finite, then $k$ contains an element $t$ transcendental over $k'$. Thus $R$ contains the subfield $k'(t)$ of rational functions in the variable $t$. So, if $R'_{\alpha}:= R_{\alpha}\otimes_{k'} k'(t)$, then $\phi_{\alpha}$ can be decomposed as follows
\[
  R_{\alpha}\to R_{\alpha}\otimes_{k'}k'(t)=R'_{\alpha}\xrightarrow{\psi_{\alpha}}R.
\]
Let $\bG'_{\alpha}=\bG_{\alpha}\times_{\spec R_{\alpha}}\spec R'_{\alpha}$, $\cE'_{\alpha}=\cE_{\alpha}\times_{\spec R_{\alpha}}\spec R'_{\alpha}$, $f'_{\alpha}=f_{\alpha}\otimes 1\in R'_{\alpha}$, then the $\bG'_{\alpha}$-bundle $\cE'_{\alpha}$ is trivial over $(R'_{\alpha})_{f'_{\alpha}}$.

Let $\mathfrak q_i=\psi_{\alpha}^{-1}(\mathfrak m_i)$ for $i=1,\dots, n$. The ring $R'_{\alpha}$ is a $k'(t)$-smooth algebra over the infinite field $k'(t)$, and the $\bG'_{\alpha}$-bundle $\cE'_{\alpha}$ is trivial over $(R'_{\alpha})_{f'_{\alpha}}$. By Lemma~\ref{lm:primemax} the $\bG'_{\alpha}$-bundle $\cE'_{\alpha}$ is trivial over $(R'_{\alpha})_{\mathfrak q_{1},\dots,\mathfrak q_{n}}$. The homomorphism $\psi_{\alpha}$ can be factored as
\[
  R'_{\alpha}\to (R'_{\alpha})_{\mathfrak q_{1},\dots,\mathfrak q_{n}}\to R.
\]
Thus the $\bG$-bundle $\cE$ is trivial over $R$.
\end{proof}
\begin{remark*}
    If $k$ is perfect, we can use it instead of $k'$, and the above proof simplifies.
\end{remark*}

\section{Reducing Theorem~\ref{th:psv} to Theorem~\ref{MainThm2}}\label{sect:reducing}
Now we assume that Theorem~\ref{MainThm2} is true. Let $U$ and $\bG$ be as in Theorem~\ref{th:psv}. Let $u_1,\ldots,u_n$ be all the closed points of $U$. Let $k(u_i)$ be the residue field of $u_i$. Consider the reduced closed subscheme $\bu$ of $U$, whose points are $u_1$, \ldots, $u_n$. Thus
\[
 \bu\cong\coprod_i\spec k(u_i).
\]
Set $\bG_\bu=\bG\times_U\bu$. By $\bG_{u_i}$ we denote the fiber of $\bG$ over $u_i$; it is a simple simply-connected algebraic group over $k(u_i)$. Let $\bu'\subset\bu$ be the subscheme of all closed points $u_i$ such that the group $\bG_{u_i}$ is isotropic. Set $\bu''=\bu-\bu'$. It is possible that $\bu'$ or $\bu''$ is empty.

\begin{proposition}\label{pr:etale}
There is a closed subscheme $Y\subset\P^1_U$ such that $Y$ is \'etale over $U$, $\bG_Y=\bG\times_UY$ is isotropic, and for all $u_i\in\bu'$ there is a $k(u_i)$-rational point $y_i\in Y$ lying over $u_i$.
\end{proposition}
\begin{proof}
If $\bu'$ is empty, we just take $Y=\emptyset$.

Otherwise, for every $u_i$ in $\bu'$ choose a proper parabolic subgroup $\bP_{u_i}$ in $\bG_{u_i}$. Let $\cP_i$ be the $U$-scheme of parabolic subgroup schemes of $\bG$ of the same type as $\bP_{u_i}$. It is a smooth projective $U$-scheme (see~\cite[Cor.~3.5, Exp.~XXVI]{SGA3-3}). The subgroup $\bP_{u_i}$ in $\bG_{u_i}$ is a $k(u_i)$-rational point~$p_i$ in the fibre of $\cP_i$ over the point $u_i$.

We claim that there is a closed subscheme $Y_i$ of $\cP_i$ such that $Y_i$ is \'etale over $U$ and $p_i\in Y_i$. Indeed, let $r$ be the dimension of $\cP_i$ over $U$ and take an embedding of $\cP_i$ into the projective space $\P^N_U=\Proj(R[x_0,\ldots,x_N])$. Let $\mathfrak m_j$ be the maximal ideal in $R$ corresponding to $u_j\in\bu$. Since $k$ is infinite, by a variant of Bertini's theorem (see~\cite[Exp.~XI, Thm.~2.1]{SGA4-3}), for each $j$ there is a sequence of homogeneous quadratic polynomials $H_1^j,\ldots,H_r^j\in(R/\mathfrak m_j)[x_0,\ldots,x_N]$ such that the subscheme $T_j$ of $\P^N_{k(u_j)}$ given by the equations $H_1^j=\ldots=H_r^j=0$ intersects the fiber of $\cP_i$ over $u_j$ transversally. Moreover, we may assume that $p_i\in T_i$. By the Chinese Remainder Theorem for each $m\in\{1,\ldots,r\}$ there is a common lift of polynomials $H_m^j$ to a quadratic polynomial $H_m\in R[x_0,\ldots,x_N]$. Let $T$ be the scheme given by $H_1=\ldots=H_r=0$. Then $Y_i:=T\cap\cP_i$ is the required subscheme. Indeed, we only need to check that $Y_i$ is \'etale over $U$. However, for every closed point of $U$ the fiber of $Y_i$ over this point is \'etale by construction. Hence, it is enough to check that $Y_i$ is flat over $U$. The flatness follows immediately from~\cite[Thm.~23.1]{MatsumuraCommRingTh}.

Now consider $Y_i$ just as a $U$-scheme and set $Y=\coprod_{u_i\in\bu'}Y_i$. Next, $\bG_{Y_i}$ is isotropic by the choice of $Y_i$. Thus $\bG_Y$ is isotropic as well.
Since the field $k$ is infinite and $Y$ is finite \'{e}tale over $U$, we can choose a closed $U$-embedding of $Y$ in $\A^1_U$. We will identify $Y$ with the image of this closed embedding. Since $Y$ is finite over $U$, it is closed in $\P_U^1$.
\end{proof}

\begin{proof}[Proof of Theorem~\ref{th:psv} assuming Theorem~\ref{MainThm2}]
Set $Z:=\{h=0\}\cup s(U)\subset\A^1_U$. Note that $\{h=0\}$ is closed in $\P^1_U$ and finite over $U$ because $h$ is monic. Further, $s(U)$ is also closed in $\P^1_U$ and finite over $U$ because it is a zero set of a degree one monic polynomial.
Thus $Z\subset\P^1_U$ is closed and finite over $U$.

Let $Y$ be as in Proposition~\ref{pr:etale}. Since $U$ is semi-local, there exists a projective transformation $\theta:\P^1_U\to\P^1_U$ such that $Z\cap\theta(Y)=\emptyset$. Thus, replacing $Y$ by $\theta(Y)$ we may assume that $Z\cap Y=\emptyset$.

Since the principal $\bG$-bundle $\cE_t$ is trivial over $(\A^1_U)_h$, and $\bG$-bundles can be glued in the Zariski topology, there exists a principal $\bG$-bundle $\cG$ over $\P^1_U$ such that\\
\indent (i) its restriction to $\A^1_U$ coincides with $\cE_t$;\\
\indent (ii) its restriction to $\P^1_U-Z$ is trivial.

Applying Theorem~\ref{MainThm2} with the above choice of $Y$ and $Z$, we see that the restriction of $\cG$ to $\P^1_U-Y$ is a trivial $\bG$-bundle. Since $s(U)$ is in $(\P^1_U-Y)\cap\A^1_U$, and $\cG|_{\A^1_U}$ coincides with $\cE_t$, we conclude that $s^*\cE_t$ is a trivial principal $\bG$-bundle over $U$.
\end{proof}

\section{Proof of Theorem~\ref{MainThm2}}\label{sect:proof2}
We will be using notation from Theorem~\ref{MainThm2}. Let $\bu$, $\bu'$, and $\bu''$ be as in Section~\ref{sect:reducing}. For $u\in\bu$ set $\bG_u=\bG|_u$.

\begin{proposition}\label{pr:trivclsdfbr}
Let $\cE$ be a $\bG$-bundle over $\P^1_U$ such that $\cE|_{\P^1_u}$ is a trivial $\bG_u$-bundle for all $u\in\bu$. Assume that there exists a closed subscheme $T$ of $\P^1_U$ finite over $U$ such that the restriction of $\cE$ to $\P^1_U-T$ is trivial. Then $\cE$ is trivial.
\end{proposition}
\begin{proof}
This follows from Proposition~9.6 of~\cite{PSV}.
\end{proof}
\begin{remark}
    The same proof goes through for any semi-simple $U$-group scheme $\bG$.
\end{remark}

\subsection{An outline of a proof of Theorem~\ref{MainThm2}}\label{sect:outline}
A detailed proof will be given in the present text below. Firstly, we give an outline of the proof.

Denote by $Y^h$ the Henselization of the pair $(\A_U^1,Y)$; it is a scheme over $\A_U^1$. We review some facts about Henselization of pairs in Section~\ref{sect:distinguishedLimit}. In particular, there exists a canonical closed embedding $s^h:Y\to Y^h$, and we set $\dot Y^h:=Y^h-s^h(Y)$. We have a natural Cartesian square (see Section~\ref{sect:gluing} for more details)
\begin{equation*}
\begin{CD}
\dot Y^h @>>> Y^h\\
@VVV @VVV\\
\P^1_U - Y @>>>\P^1_U.
\end{CD}
\end{equation*}
This square can be used to glue principal bundles. In particular, if $\cG'$ is a $\bG$-bundle over $\P^1_U-Y$, then by $\Gl(\cG',\phi)$ we denote the $\bG$-bundle over $\P_U^1$ obtained by gluing $\cG'$ with the trivial $\bG$-bundle $\bG\times_U Y^h$ via a $\bG$-bundle isomorphism $\phi:\bG\times_U\dot Y^h\to\cG'|_{\dot Y^h}$.

Similarly, set $Y_\bu:=Y\times_U\bu$ and denote by $Y_\bu^h$ the Henselization of the pair $(\A_\bu^1,Y_\bu)$, let
$s_\bu^h:Y_\bu\to Y_\bu^h$ be the closed embedding. Set $\dot Y_\bu^h:=Y_\bu^h-s_\bu(Y_\bu)$. Let $\cG'_\bu$ be a $\bG_\bu$-bundle over $\P^1_\bu-Y_\bu$, where $\bG_\bu:=\bG\times_U\bu$. Denote by $\Gl_\bu(\cG'_\bu,\phi_\bu)$ the $\bG_\bu$-bundle over $\P_\bu^1$ obtained by gluing $\cG'_\bu$ with the trivial bundle $\bG_\bu\times_\bu Y_\bu^h$ via a $\bG_\bu$-bundle isomorphism $\phi_\bu:\bG_\bu\times_\bu\dot Y_\bu^h\to\cG'_\bu|_{\dot Y_\bu^h}$.

We will prove in Section~\ref{sect:presentation} that the restriction of the $\bG$-bundle $\cG$ to $Y^h$ is trivial, so $\cG$ can be presented in the form $\Gl(\cG',\phi)$, where $\cG'=\cG|_{\P^1_U-Y}$. The idea is to show that

\vskip4pt
\newlength{\shortwidth}
\setlength{\shortwidth}{\textwidth}
\addtolength{\shortwidth}{-.5cm}

\noindent ($*$) \emph{\hfill \parbox{\shortwidth}{There is an element $\alpha\in\bG(\dot Y^h)$ such that the $\bG_\bu$-bundle $\Gl(\cG',\phi\circ\alpha)|_{\P^1_\bu}$ is trivial (here $\alpha$ is regarded as an automorphism of the $\bG$-bundle $\bG\times_U\dot Y^h$ given by right translation action} of $\alpha$).}

\vskip4pt

If we find $\alpha$ satisfying condition ($*$), then Proposition~\ref{pr:trivclsdfbr}, applied to $T=Y\cup Z$, shows that the $\bG$-bundle $\Gl(\cG',\phi\circ\alpha)$ is trivial over $\P^1_U$. On the other hand, its restriction to $\P^1_U-Y$ coincides with the $\bG$-bundle $\cG'=\cG|_{\P^1_U-Y}$. \emph{Thus $\cG|_{\P^1_U-Y}$ is a trivial $\bG$-bundle}.

To prove ($*$), one should show that\\
\indent (i) the bundle $\cG|_{\P^1_\bu-Y_\bu}$ is trivial;\\
\indent (ii) each element $\gamma_\bu\in\bG_\bu(\dot Y_\bu^h)$ can be written in the form
\[
 \alpha|_{\dot Y_\bu^h}\cdot\beta_\bu|_{\dot Y_\bu^h}
\]
for certain elements $\alpha\in\bG(\dot Y^h)$ and $\beta_\bu\in\bG_\bu(Y_\bu^h)$.

If we succeed in showing that (i) and (ii) above hold, then we proceed as follows. Present the $\bG$-bundle $\cG$ in the form $\Gl(\cG',\phi)$ as above. Observe that
\[
 \Gl(\cG',\phi)|_{\P^1_\bu}\cong\Gl_\bu(\cG'_\bu,\phi_\bu),
\]
where $\cG'_\bu:=\cG'|_{\P^1_\bu-Y_\bu}$, $\phi_\bu:=\phi|_{\bG_\bu\times_\bu\dot Y_\bu^h}$.

Using property (i), find an element $\gamma_\bu\in\bG_\bu(\dot Y_\bu^h)$ such that the $\bG_\bu$-bundle $\Gl_\bu(\cG'_\bu,\phi_\bu\circ\gamma_\bu)$ is trivial. For this $\gamma_\bu$ find elements $\alpha$ and $\beta_\bu$ as in (ii). Finally take the $\bG$-bundle $\Gl(\cG',\phi\circ\alpha)$. Then its restriction to $\P_\bu^1$ is trivial. Indeed, one has a chain of $\bG_\bu$-bundle isomorphisms
\begin{multline*}
\Gl(\cG',\phi \circ \alpha)|_{\P^1_\bu}\cong\Gl_\bu(\cG'_\bu,\phi_\bu \circ \alpha|_{\dot Y_\bu^h})
\cong\\
\Gl_\bu(\cG_\bu',\phi_\bu \circ
\alpha|_{\dot Y_\bu^h}\circ \beta_\bu|_{\dot Y_\bu^h})=
\Gl_\bu(\cG_\bu', \phi_\bu \circ \gamma_\bu),
\end{multline*}
which is trivial by the very choice of $\gamma_\bu$. Thus, ($*$) will be achieved.

Let us prove (i) and (ii). If $u\in\bu'$, then there is a $k(u)$-rational point in $Y_u:=\P^1_u\cap Y$. Hence
the $\bG_u$-bundle $\cG_u:=\cG|_{\P_u^1}$ is trivial over $\P^1_u-Y_u$ (see~\cite[Cor.~3.10(a)]{GilleTorseurs}).
If $u\in\bu''$, then $\bG_u$ is anisotropic and $\cG_u$ is trivial even over $\P^1_u$ (again, by~\cite[Cor.~3.10(a)]{GilleTorseurs}). Thus $\cG|_{\P^1_\bu-Y_\bu}$ is trivial. So, (i) is achieved.

To achieve (ii) recall that for a domain $A$, its fraction field $L$, and a simple group scheme $\bH$ over $A$, having a parabolic subgroup scheme $\bP$, one can form a subgroup $\bE(L)$ of ``elementary matrices' in $\bH(L)$. It is known (see~\cite[Fait~4.3, Lemma~4.5]{Gille:BourbakiTalk}) that if $A$ is a Henselian discrete valuation ring and $\bH$ is simply-connected, then every element $\gamma\in\bH(L)$ can be written in the form $\gamma=\alpha\cdot\beta$, where $\alpha\in\bE(L)$ and $\beta\in\bH(A)$. Applying this observation in our context, we see that $\gamma_\bu\in\bG_\bu(\dot Y_\bu^h)$ can be written in the form $\gamma_\bu=\alpha_\bu\cdot\beta_\bu|_{\dot Y_\bu^h}$, where $\beta_\bu\in\bG_\bu(Y_\bu^h)$ and $\alpha_\bu\in\bE(\dot Y_\bu^h)$. It remains to observe that the natural homomorphism $\bE(\dot Y^h)\to\bE(\dot Y_\bu^h)$ is surjective, since $\dot Y_\bu^h$ is a closed subscheme of the affine scheme
$\dot Y^h$, and so (ii) is achieved.

A realization of this plan in details is given below in the paper.

\subsection{Henselization of commutative rings}
For a commutative ring $A$ we denote by $\Rad(A)$ its Jacobson ideal. One can find the following definition in~\cite[Sect.~0]{Gabber} (see also~\cite[Chapter~11]{RaynaudHenselian}).
\begin{definition}
If $I$ is an ideal in a commutative ring $A$, then the pair $(A,I)$ is called \emph{Henselian\/}, if $I\subset\Rad(A)$ and for every two relatively prime monic polynomials $\bar g, \bar h\in\bar A[t]$, where $\bar A=A/I$, and monic lifting $f\in A[t]$ of $\bar g\bar h$, there exist monic liftings $g,h\in A[t]$ such that $f=gh$. (Two polynomials are called relatively prime, if they generate the unit ideal.)
\end{definition}

\begin{lemma}\label{lm:otherhenselian}
A pair $(A,I)$ is Henselian if and only if for every \'etale $A$-algebra $A'$ and every $\sigma\in\Hom_{A-Alg}(A',A/I)$ there is a unique $\bar\sigma\in\Hom_{A-Alg}(A',A)$ that lifts $\sigma$.
\end{lemma}
\begin{proof}
See~\cite[Sect.~0]{Gabber}.
\end{proof}

\begin{lemma}\label{basechhens}
Let $(A,I)$ be a Henselian pair with a semi-local ring $A$ and $J\subset A$ be an ideal. Then the pair $(A/J,(I+J)/J)$ is Henselian.
\end{lemma}
\begin{proof}
Clearly $(I+J)/J\subset\Rad(A/J)$. Now let $\bar g,\bar h\in (A/(I+J))[t]$ be two relatively prime monic polynomials and let $f\in(A/J)[t]$ be a monic polynomial such that $f\bmod(I+J)/J=\bar g\bar h\in (A/(I+J))[t]$.

We claim that there exist relatively prime monic liftings of $\bar g$ and $\bar h$ to $(A/I)[t]$. Indeed, let $\fm_1$, \ldots, $\fm_n$ be all the maximal ideals of $A/I$ not containing $(I+J)/I$ (recall that $A$ is semi-local). By the Chinese remainder theorem we can find monic $\bar G,\bar H\in(A/I)[t]$ such that
\[
\begin{split}
\bar G\bmod (I+J)/I=\bar g, &\quad \bar G\bmod\fm_i=t^{\deg\bar g}\mbox{ for } i=1,\ldots,n,\\
\bar H\bmod (I+J)/I=\bar h, &\quad \bar H\bmod\fm_i=t^{\deg\bar h}-1\mbox{ for } i=1,\ldots,n.
\end{split}
\]
Then $\bar G$ and $\bar H$ are relatively prime. The ring homomorphism
\[
A\to(A/I)\times_{A/(I+J)}(A/J)
\]
is surjective. Thus there exists a monic polynomial $F\in A[t]$ such that $F\bmod I=\bar G\bar H$ and $F\bmod J=f$.

The pair $(A,I)$ is Henselian. Thus there exist monic liftings $G,H\in A[t]$ of $\bar G,\bar H$ such that $F=GH$. Let $g=G\bmod J\in(A/J)[t]$ and $h=H\bmod J\in(A/J)[t]$. Clearly, $g$ and $h$ are monic polynomials in $(A/J)[t]$, $f=gh\in(A/J)[t]$. And finally, $g\bmod(I+J)/J=\bar g$, $h\bmod(I+J)/J=\bar h$ in $(A/(I+J))[t]$.
Whence the Lemma.
\end{proof}

One can find the following definition in~\cite[Sect.~0]{Gabber}.
\begin{definition}\label{def:henzelisation}
The \emph{Henselization\/} of a pair $(A,I)$ is the pair $(A_I^h,I^h)$ (over $(A,I)$) defined as follows
\[
(A_I^h,I^h):=\text{the filtered inductive limit over the category $\mathcal N$ of }(A',\Ker(\sigma)),
\]
where $\mathcal N$ is the filtered category of pairs $(A',\sigma)$ such that $A'$ is an \'{e}tale $A$-algebra and $\sigma\in\Hom_{A-alg}(A',A/I)$.
\end{definition}
Note that the category $\mathcal N$ is filtered because finite direct limits preserve \'etalness.

\subsection{Henselization of affine pairs}\label{sect:distinguishedLimit} Let us translate the previous section in the geometric language. Let $S=\spec A$ be a scheme and $T=\spec(A/I)$ be a closed subscheme. Then we define a category $\wN(S,T)$ whose objects are triples $(W,\pi:W\to S,s: T\to W)$ satisfying the following conditions:\\
\indent (i) $W$ is affine;\\
\indent (ii) $\pi$ is an \'{e}tale morphism;\\
\indent (iii) $\pi\circ s$ coincides with the inclusion $T\hookrightarrow S$ (thus $s$ is a closed embedding).

A morphism from $(W,\pi,s)$ to $(W',\pi',s')$ in this category is a morphism $\rho:W\to W'$ such that $\pi'\circ\rho=\pi$ and $\rho\circ s=s'$. Note that such $\rho$ is automatically \'etale by~\cite[Cor.~17.3.5]{EGAIV.4}.

Consider the functor from $\wN(S,T)$ to the category of $S$-schemes, sending $(W,\pi,s)$ to $(W,\pi)$. This functor has a projective limit $(T^h,\pi^h)$. In the notation of the previous section we have $T^h=\spec A^h_I$ and $\pi^h:T^h\to S$ is induced by the structure of an $A$-algebra on $A^h_I$. We also get a closed $S$-embedding $s^h:T\to T^h$, that is, $\pi^h\circ s^h$ coincides with the inclusion $T\hookrightarrow S$. We call $(T^h,\pi^h,s^h)$ \emph{the Henselization of the pair $(S,T)$} (cf. Definition~\ref{def:henzelisation}). Note that the pair $(T^h,s^h(T))$ is Henselian, which means that for any affine \'etale morphism $\pi:Z\to T^h$, any section $\sigma$ of $\pi$ over $s^h(T)$ uniquely extends to a section of $\pi$ over $T^h$; this follows from Lemma~\ref{lm:otherhenselian}.

Denote by $\Nneib(S,T)$ the full subcategory of $\wN(S,T)$ consisting of triples $(W,\pi,s)$ such that\\
\indent (iv) the schemes $(\pi)^{-1}(T)$ and $s(T)$ coincide.
\begin{remark*}
    Let $(W,\pi,s)$ and $(W',\pi',s')$ be objects of $\Nneib(S,T)$. Let $\rho:W\to W'$ be a morphism such that $\pi'\circ\rho=\pi$. Then it is easy to see that $\rho\circ s=s'$ so that $\rho$ is a morphism in $\Nneib(S,T)$. (Again, $\rho$ is automatically \'etale.)
\end{remark*}

\begin{lemma}
$\Nneib(S,T)$ is co-final in $\wN(S,T)$.
\end{lemma}
\begin{proof}
We need to check that for an object $(W,\pi,s)$ of $\wN(S,T)$ there is an object $(W',\pi',s')$ of $\Nneib(S,T)$ and a morphism $(W',\pi',s')\to(W,\pi,s)$. Let $\pi_T:(\pi)^{-1}(T)\to T$ be the base-changed morphism, which is \'etale. It follows from (iii) that $s$ is a section of $\pi_T$. As was already mentioned above, a section of an \'etale morphism is \'etale by~\cite[Cor.~17.3.5]{EGAIV.4}. Thus $s$ is both an open and a closed embedding, and we have a disjoint union decomposition $(\pi)^{-1}(T)=s(T)\coprod T_0$ for a scheme $T_0$. All our schemes are affine, so there is a regular function $f$ on $W$ such that $f=0$ on $T_0$ and $f=1$ on $s(T)$.

Set $W'=W-\{f=0\}$, $\pi'=\pi|_{W'}$, $s'=s$. Then $W'$ is affine; thus $(W',\pi',s')\in\Nneib(S,T)$, and we have an obvious morphism $(W',\pi',s')\to(W,\pi,s)$.
\end{proof}
The lemma implies that the category $\Nneib(S,T)$ is co-filtered, and that the Henselization can be computed by taking the limit over $\Nneib(S,T)$, instead of $\wN(S,T)$. It is now easy to check that if $(T^h,\pi^h,s^h)$ is the Henselization of $(S,T)$, then $(\pi^h)^{-1}(T)=s^h(T)$.

Note the two following properties of Henselization of affine pairs.
\begin{lemma}\label{lm:propi}
Let $T$ be a semi-local scheme. Then the Henselization commutes with restriction to closed subschemes. In more detail, if $S'\subset S$ is a closed subscheme, then we get a base change functor $\wN(S,T)\to\wN(S',T\times_SS')$. This functor yields a morphism
$(T\times_SS')^h\to T^h\times_SS'$. This morphism is an isomorphism and the canonical section $s':T\times_SS'\to(T\times_SS')^h$ coincides under this identification with
\[
    s\times_S\Id_{S'}:T\times_SS'\to T^h\times_SS'.
\]
\end{lemma}
\begin{proof}[Sketch of a proof]
Let us construct a morphism in the opposite direction. Since $T$ is semi-local, $T^h$ is also semi-local (the proof is straightforward). Therefore by Lemma~\ref{basechhens} the pair $(T^h\times_SS',s(T)\times_SS')$ is Henselian.

Let $(W,\pi,s)\in\wN(S',T\times_SS')$. From $\pi$ by a base change we get an \'etale morphism $\tilde\pi:(T^h\times_SS')\times_{S'}W\to T^h\times_SS'$. This morphism has an obvious section over $s(T)\times_SS'$. Since the pair $(T^h\times_SS',s(T)\times_SS')$ is Henselian, this section extends uniquely to a section of $\tilde\pi$ over $T^h\times_SS'$, which, in turn, gives a morphism $T^h\times_SS'\to W$. These morphisms give the desired morphism $T^h\times_SS'\to(T\times_SS')^h$.
\end{proof}
\begin{lemma}\label{lm:propii}
If $T=\coprod_i T_i$ is a disjoint union, then $T^h=\coprod_i T_i^h$.
\end{lemma}
\begin{proof}[Sketch of a proof]
Note that the functor from $\prod_i\wN(S,T_i)$ to $\wN(S,T)$, sending a collection of schemes to their disjoint union, is co-final.
\end{proof}

\subsection{Gluing principal $\bG$-bundles}\label{sect:gluing}
Recall that $U=\spec R$, where $R$ is the semi-local ring of finitely many closed points on an irreducible smooth affine variety over an infinite field $k$. Also, $\bG$ is a simple simply-connected group scheme over $U$, and $Y$ is a closed subscheme of $\P_U^1$ \'etale over $U$. We may assume that $Y\subset\A_U^1$ (otherwise, just change the coordinate on $\P_U^1$). We will apply the Henselization discussed above to $S=\A_U^1$, $T=Y$. Thus we have an affine scheme $Y^h$ with a projection $\pi^h:Y^h\to Y$ and a section $s^h:Y\to Y^h$. Set $\dot Y^h=Y^h-s(Y)$.
\begin{lemma}\label{lm:affine}
    If $(W,\pi,s)\in\Nneib(\A_U^1,Y)$, then $s(Y)$ is a principal divisor in $W$ and therefore $W-s(Y)$ is affine.
\end{lemma}
\begin{proof}
Since $U$ is a regular semi-local ring, $Y$ is a principal divisor in $\A_U^1$. Thus $s(Y)=(\pi)^{-1}(Y)$ is also a principal divisor in the affine scheme $W$.
\end{proof}

Let us make a general remark. Let $\cF$ be a $\bG$-bundle over a $U$-scheme $T$. By definition, a trivialization of $\cF$ is a $\bG$-equivariant isomorphism $\bG\times_UT\to\cF$. Equivalently, it is a section of the projection $\cF\to T$. If $\phi$ is such a trivialization and $f:T'\to T$ is a $U$-morphism, we get a trivialization $f^*\phi$ of $f^*\cF$. Sometimes we denote this trivialization by $\phi|_{T'}$. We also sometimes call a trivialization of $f^*\cF$ \emph{a trivialization of $\cF$ on $T'$}.

We will recall some consequences of Nisnevich descent. Let $in:\A^1_U\hookrightarrow\P^1_U$ be the standard inclusion.
For each object $(W,\pi,s)$ in $\Nneib(\A_U^1,Y)$ there is an elementary distinguished square (see~\cite[Def.~2.1]{VoevodskyCongress})
\begin{equation}\label{eq:distinguished}
\begin{CD}
W-s(Y) @>>>W\\
@VVV @VV in\circ\pi V\\
\P_U^1 - Y @>>>\P^1_U.
\end{CD}
\end{equation}
It is used here that $Y$ is closed in $\P^1_U$.

The elementary distinguished square (\ref{eq:distinguished}) can be used to construct principal $\bG$-bundles over $\P_U^1$ via Nisnevich descent. In particular, one can glue a principal bundle over $\P_U^1-Y$ with a trivial principal bundle over $W$ via an isomorphism on $W-s(Y)$. More precisely, let $\cA(W,\pi,s)$ be the category of pairs $(\cE,\phi)$, where $\cE$ is a $\bG$-bundle over $\P_U^1$, $\phi$ is a trivialization of $\cE|_W:=(in\circ\pi)^*\cE$. A morphism between $(\cE,\phi)$ and $(\cE',\phi')$ is an isomorphism $\cE\to\cE'$ compatible with trivializations.

Similarly, let $\cB(W,\pi,s)$ be the category of pairs $(\cE,\phi)$, where $\cE$ is a $\bG$-bundle over $\P_U^1-Y$, $\phi$ is a trivialization of $\cE|_{W-s(Y)}$.

\begin{lemma}\label{lm:groupoids}
The categories $\cA(W,\pi,s)$ and $\cB(W,\pi,s)$ are groupoids whose objects have no non-trivial automorphisms.
\end{lemma}
\begin{proof}
    It is obvious that the categories are groupoids. Consider an object $(\cE,\phi)\in\cA(W,\pi,s)$. Let $\alpha$ be an automorphism of $\cE$ such that $\alpha|_W=\Id_{\cE|_W}$. We need to show that $\alpha=\Id_\cE$. This follows immediately from the fact that the $\Aut(\cE)$ is represented by a scheme affine over $\P_U^1$ (see the proof of Corollary~\ref{corollary}), while $\P_U^1$ is irreducible. The statement for $\cB(W,\pi,s)$ is proved similarly.
\end{proof}

Consider the restriction functor $\Phi:\cA(W,\pi,s)\to\cB(W,\pi,s)$. The following proposition is a version of Nisnevich descent.
\begin{proposition}\label{pr:gluing}
    The functor $\Phi$ is an equivalence of categories.
\end{proposition}
\begin{proof}
Let us prove that $\Phi$ is essentially surjective. Let $(\cE,\phi)$ be an object of $\cB(W,\pi,s)$, set $\cE'=\cE|_{\A_U^1-Y}$.
By Lemma~\ref{lm:affine} and~\cite[Prop.~2.6(iv)]{C-TO} there is a $\bG$-bundle $\cE''$ over $\A_U^1$, a trivialization $\phi''$ of $\cE''$ on $W$, and an isomorphism
\[
    \cE''|_{\A_U^1-Y}\to\cE'=\cE|_{\A_U^1-Y}
\]
compatible with the trivializations on $W-s(Y)$. We can use this isomorphism to glue $\cE$ with $\cE''$ to make a $\bG$-bundle $\tilde\cE$ over $\P_U^1$ (gluing in the Zariski topology). The trivialization $\phi''$ gives rise to a trivialization $\tilde\phi$ of $\tilde\cE$ on $W$. Clearly, $\Phi(\tilde\cE,\tilde\phi)\cong(\cE,\phi)$.

It follows immediately from Lemma~\ref{lm:groupoids} that $\Phi$ is faithful. It remains to show that $\Phi$ is full. Let $(\cE,\phi)$ and $(\cE',\phi')$ be objects of $\cA(W,\pi,s)$. Let $\alpha$ be a morphism from $\Phi(\cE,\phi)$ to $\Phi(\cE',\phi')$. We need to show that $\alpha$ is of the form $\Phi(\beta)$.

Recall that the presheaf $\Iso(\cE,\cE')$ is represented by a $\P^1_U$-scheme (see the proof of Corollary~\ref{corollary}), so, in particular, it is a sheaf in the Nisnevich topology. Thus, since~\eqref{eq:distinguished} is an elementary distinguished square, to give a section of $\Iso(\cE,\cE')$ over $\P^1_U$ is the same as to give sections over $\P^1_U-Y$ and over $W$ that coincide over $W-s(Y)$ (see~\cite[Sect.~3, Prop.~1.3]{MorelVoevodsky}).

Note that $\alpha$ gives a section of $\Iso(\cE,\cE')$ over $\P^1_U-Y$, while $\phi'\circ\phi^{-1}$ is a section over $W$. By definition of $\cB(W,\pi,s)$ these sections coincide on $W-s(Y)$, so we obtain a section $\beta$ of $\Iso(\cE,\cE')$ over $\P^1_U$.
By construction $\beta$ is a morphism in $\cA(W,\pi,s)$ and $\Phi(\beta)=\alpha$.
\end{proof}

The main Cartesian square we will work with is
\begin{equation}
\begin{CD}\label{eq:distinguishedLimit}
\dot Y^h @>>> Y^h\\
@VVV @VV{in\circ\pi^h}V\\
\P^1_U - Y @>>>\P^1_U.
\end{CD}
\end{equation}
\begin{proposition}\label{pr:affine}\stepzero
\noindstep\label{b} $\dot Y^h$ is the projective limit of $W-s(Y)$ over $\Nneib(\A_U^1,Y)$.\\
\noindstep\label{c} $\dot Y^h$ is an affine scheme.
\end{proposition}
\begin{proof}
Part~\eqref{b} follows from the definition of projective limit and the equality $s^h(Y)=(\pi^h)^{-1}(Y)$. Part~\eqref{c} follows from Lemma~\ref{lm:affine}, part~\eqref{b}, and~\cite[Prop.~8.2.3]{EGAIV-3}.
\end{proof}

Let $\cA$ be  the category  of pairs $(\cE,\psi)$, where $\cE$ is a $\bG$-bundle over $\P_U^1$, $\psi$ is a trivialization of $\cE|_{Y^h}:=(in\circ\pi^h)^*\cE$. A morphism between $(\cE,\psi)$ and $(\cE',\psi')$ is an isomorphism $\cE\to\cE'$ compatible with trivializations.

Similarly, let $\cB$ be the category of pairs $(\cE,\psi)$, where $\cE$ is a $\bG$-bundle over $\P_U^1-Y$, $\psi$ is a trivialization of $\cE|_{\dot Y^h}$.

\begin{lemma}\label{lm:groupoidAB}
The categories $\cA$ and $\cB$ are groupoids whose objects have no non-trivial automorphisms.
\end{lemma}
\begin{proof}
    It is obvious that the categories are groupoids. Note that for a $\bG$-bundle $\cE$ we have
    \[
        (\Aut(\cE))(Y^h)=\lim_{(W,\pi,s)\in\Nneib(\A_U^1,Y)}(\Aut(\cE))(W).
    \]
    Thus an automorphism of $\cE$ that is equal to the identity on $Y^h$ is equal to the identity on some $W$ with $(W,\pi,s)\in\Nneib(\A_U^1,Y)$. Now Lemma~\ref{lm:groupoids} shows that such an automorphism is equal to the identity. The statement for objects of $\cB$ is proved similarly in view of Proposition~\ref{pr:affine}\eqref{b}.
\end{proof}

Consider the restriction functor $\Psi:\cA\to\cB$.
\begin{proposition}\label{pr:gluing2}
    The functor $\Psi$ is an equivalence of categories.
\end{proposition}
\begin{proof}
Let us prove that $\Psi$ is essentially surjective; let $(\cE,\psi)\in\cB$. Then using Lemma~\ref{lm:affine} and Proposition~\ref{pr:affine}\eqref{b}, we can find $(W,\pi,s)\in\Nneib(\A_U^1,Y)$ and a trivialization $\phi$ of
$\cE$ on $W-s(Y)$ such that $\phi|_{\dot Y^h}=\psi$. By proposition~\ref{pr:gluing} there is $(\tilde\cE,\tilde\phi)\in\cA(W,\pi,s)$ such that $\Phi(\tilde\cE,\tilde\phi)\cong(\cE,\phi)$. Then
\[
\Psi(\tilde\cE,\tilde\phi|_{Y^h})=(\tilde\cE|_{\P_U^1-Y},\tilde\phi|_{\dot Y^h})
\cong\left(\cE,\phi|_{\dot Y^h}\right)=(\cE,\psi).
\]

It follows immediately from Lemma~\ref{lm:groupoidAB} that $\Psi$ is faithful. It remains to show that $\Psi$ is full. Let $(\cE,\psi)$ and $(\cE',\psi')$ be objects of $\cA$. Let $\alpha$ be a morphism from $\Psi(\cE,\psi)$ to $\Psi(\cE',\psi')$. We need to show that $\alpha$ is of the form $\Psi(\beta)$.

We can find $(W,\pi,s)\in\Nneib(\A_U^1,Y)$ and trivializations $\phi$ and $\phi'$ of $\cE$ and $\cE'$ respectively on $W$
such that $\phi|_{Y^h}=\psi$, $\phi'|_{Y^h}=\psi'$. Using Proposition~\ref{pr:affine}\eqref{b} it is easy to check that the restriction morphism $\Iso(\cE,\cE')(W-s(Y))\to\Iso(\cE,\cE')(\dot Y^h)$ is injective. Thus $\alpha$ is a morphism in $\cB(W,\pi,s)$ from $\Phi(\cE,\phi)$ to $\Phi(\cE',\phi')$. By Proposition~\ref{pr:gluing} there is a morphism $\beta$ from $(\cE,\phi)$ to $(\cE',\phi')$ such that $\Phi(\beta)=\alpha$. Then $\beta$ is also a morphism in $\cA$ from $(\cE,\psi)$ to $(\cE',\psi')$ and $\Psi(\beta)=\alpha$.
\end{proof}

\begin{construction}\label{equally_well}
By Proposition~\ref{pr:gluing2} we can choose a functor quasi-inverse to $\Psi$. Fix such a functor $\Theta$. Let $\Lambda$ be the forgetful functor from $\cA$ to the category of $\bG$-bundles over $\P_U^1$. For $(\cE,\psi)\in\cB$ set
\[
    \Gl(\cE,\psi)=\Lambda(\Theta(\cE,\psi)).
\]
By construction $\Gl(\cE,\psi)$ comes with a prescribed trivialization over $Y^h$.

Conversely, if $\cE$ is a principal $\bG$-bundle over $\P_U^1$ such that its restriction to $Y^h$ is trivial, then $\cE$ can be represented as $\Gl(\cE',\psi)$, where $\cE'=\cE|_{\P_U^1-Y}$, $\psi$ is a trivialization of $\cE'$ on $\dot Y^h$.
\end{construction}

Let $\bu$ be as in Section~\ref{sect:reducing}, $Y_\bu:=Y\times_U\bu$. Let $(Y_\bu^h,\pi^h_\bu,s^h_\bu)$ be the Henselization of $(\A^1_\bu,Y_\bu)$. Using Lemma~\ref{lm:propi}, we get an identification $Y_\bu^h=Y^h\times_U\bu$. Thus we have a closed embedding $Y_\bu^h\to Y^h$. Set $\dot Y_\bu^h=Y_\bu^h-s_\bu(Y_\bu)$. We get a closed embedding $\dot Y_\bu^h\to\dot Y^h$. Thus the pull-back of the Cartesian square (\ref{eq:distinguishedLimit}) by means of the closed embedding $\bu\hookrightarrow U$ has the form
\[
\begin{CD}
\dot Y_\bu^h @>>>Y_\bu^h\\
@VVV @VV in_\bu\circ\pi_\bu^h V\\
\P^1_\bu - Y_\bu @>>>\P^1_\bu,
\end{CD}
\]
where $in_\bu:\A_\bu^1\to\P_\bu^1$ is the standard embedding. Similarly to the above, let $\cA_\bu$ be  the category  of pairs $(\cE_\bu,\psi_\bu)$, where $\cE_\bu$ is a $\bG_\bu$-bundle over $\P_\bu^1$, $\psi_\bu$ is a trivialization of $\cE|_{Y_\bu^h}$. A morphism between $(\cE_\bu,\psi_\bu)$ and $(\cE'_\bu,\psi'_\bu)$ is an isomorphism $\cE_\bu\to\cE'_\bu$ compatible with trivializations. Let $\cB_\bu$ be the category of pairs $(\cE_\bu,\psi_\bu)$, where $\cE_\bu$ is a $\bG_\bu$-bundle over $\P_\bu^1-Y_\bu$, $\psi_\bu$ is a trivialization of $\cE_\bu|_{\dot Y_\bu^h}$. We have an obvious restriction functor $\Psi_\bu:\cA_\bu\to\cB_\bu$ and, similarly to Proposition~\ref{pr:gluing2}, we show that $\Psi_\bu$ is an equivalence of categories.

Next, we have obvious restriction functors $R_\cA:\cA\to\cA_\bu$ and $R_\cB:\cB\to\cB_\bu$ and the diagram
\begin{equation}\label{CD:pullback}
\begin{CD}
\cA @>R_\cA>> \cA_\bu\\
@V \Psi VV @VV \Psi_\bu V\\
\cB @>R_\cB>> \cB_\bu
\end{CD}
\end{equation}
commutes in the sense that the functors $\Psi_\bu\circ R_\cA$ and $R_\cB\circ\Psi$ are isomorphic.

Let $\Theta_\bu$ be a functor quasi-inverse to $\Psi_\bu$ and $\Lambda_\bu$ be the forgetful functor from $\cA_\bu$ to the category of $\bG_\bu$-bundles over $\P_\bu^1$. Let $\cE_\bu$ be a principal $\bG_\bu$-bundle over $\P^1_\bu-Y_\bu$ and $\psi_\bu$ be a trivialization of $\bG_\bu$ on $\dot Y_\bu^h$. Set $\Gl_\bu(\cE_\bu,\psi_\bu)=\Lambda_\bu(\Theta_\bu(\cE_\bu,\psi_\bu))$.

\begin{lemma}\label{basechange_limit}
Let $(\cE,\psi)\in\cB$, and let $\Gl(\cE,\psi)$ be the $\bG$-bundle obtained by Construction~\ref{equally_well}. Then
\[
 \Gl_\bu(\cE|_{\P^1_\bu-Y_\bu},\psi|_{\dot Y_\bu^h}) \ \text{and} \ \Gl(\cE,\psi)|_{\P^1_\bu}
\]
are isomorphic as $\bG_\bu$-bundles over $\P^1_\bu$.
\end{lemma}
\begin{proof}
By definition of $\Gl$ we have
\[
    \Theta(\cE,\psi)=(\Gl(\cE,\psi),\sigma),
\]
where $\sigma$ is the canonical trivialization of $\Gl(\cE,\psi)$ on $Y^h$. Similarly,
\[
    \Theta_\bu(\cE|_{\P^1_\bu-Y_\bu},\psi|_{\dot Y_\bu^h})=(\Gl_\bu(\cE|_{\P^1_\bu-Y_\bu},\psi|_{\dot Y_\bu^h}),\sigma_\bu),
\]
where $\sigma_\bu$ is the canonical trivialization of $\Gl_\bu(\cE|_{\P^1_\bu-Y_\bu},\psi|_{\dot Y_\bu^h})$ on $Y_\bu^h$. Thus (since $\Psi_\bu$ is an equivalence of categories) it suffices to check that
\[
    \Psi_\bu\bigl(R_\cA(\Theta(\cE,\psi))\bigr)\cong\Psi_\bu(\Theta_\bu(\cE|_{\P^1_\bu-Y_\bu},\psi|_{\dot Y_\bu^h})).
\]
In fact, both sides are isomorphic to $(\cE|_{\P^1_\bu-Y_\bu},\psi|_{\dot Y_\bu^h})$ because diagram~\eqref{CD:pullback} is commutative.
\end{proof}

\begin{lemma}\label{coboundary_limit}
For any $(\cE_\bu,\psi_\bu)\in\cB_\bu$ and any $\beta_\bu\in\bG_\bu(Y_\bu^h)$ the $\bG_\bu$-bundles
\[
 \Gl_\bu(\cE_\bu,\psi_\bu)\ \text{and} \ \Gl_\bu(\cE_\bu,\psi_\bu\circ\beta_\bu|_{\dot Y_\bu^h})
\]
are isomorphic (here $\beta_\bu|_{\dot Y_\bu^h}$ is regarded as an automorphism of the $\bG_\bu$-bundle $\bG_\bu\times_\bu\dot Y_\bu^h$ given by the right translation action).
\end{lemma}
\begin{proof}
Denote by $\sigma_\bu$ and $\tau_\bu$ the canonical trivializations on $Y_\bu^h$ of $\Gl_\bu(\cE_\bu, \psi_\bu)$ and $\Gl_\bu(\cE_\bu,\psi_\bu \circ \beta_\bu|_{\dot Y^h_\bu})$ respectively. It is straightforward to check that $(\cE_\bu,\psi_\bu)$
is isomorphic in $\cB_\bu$ to both $\Psi_\bu(\Gl_\bu(\cE_\bu,\psi_\bu),\sigma_\bu)$ and
$\Psi_\bu(\Gl_\bu(\cE_\bu,\psi_\bu\circ\beta_\bu|_{\dot Y^h_\bu}),\tau_\bu\circ\beta_\bu^{-1})$.

Since $\Psi_\bu$ is an equivalence of categories, we conclude that $(\Gl_\bu(\cE_\bu, \psi_\bu), \sigma_\bu)$
and $(\Gl_\bu(\cE_\bu,\psi_\bu\circ\beta_\bu|_{\dot Y^h_\bu}),\tau_\bu\circ\beta_\bu^{-1})$ are isomorphic in $\cA_\bu$. Applying the functor $\Lambda_\bu$, we see that the $\bG_\bu$-bundles $\Gl_\bu(\cE_\bu, \psi_\bu)$ and
$\Gl_\bu(\cE_\bu,\psi_\bu \circ \beta_\bu|_{\dot Y^h_\bu})$ are isomorphic.
\end{proof}

\subsection{Proof of Theorem~\ref{MainThm2}: presentation of $\cG$ in the form $\Gl(\cG',\phi)$}\label{sect:presentation}
Let $U$, $\bG$, $Z$, and $\cG$ be as in Theorem~\ref{MainThm2}. We may assume that $Z\subset\A_U^1$.

\begin{proposition}\label{presentation}
The $\bG$-bundle $\cG$ over $\P^1_U$ is of the form $\Gl(\cG',\phi)$ for the $\bG$-bundle $\cG':=\cG|_{\P^1_U-Y}$ and a
trivialization $\phi$ of $\cG'$ over $\dot Y^h$.
\end{proposition}
\begin{proof}
In view of Construction~\ref{equally_well}, it is enough to prove that the restriction of the principal $\bG$-bundle $\cG$ to $Y^h$ is trivial. Let us choose a closed subscheme $Z'\subset\A^1_U$ such that $Z'$ contains $Z$, $Z'\cap Y=\emptyset$, and $\A^1_U-Z'$ is affine. Then $\A^1_U-Z'$ is an affine neighborhood of $Y$. Thus, the Henselization of the pair $(\A^1_U-Z',Y)$ coincides with the Henselization of the pair $(\A^1_U,Y)$. Since $\cG$ is trivial over $\A^1_U-Z'$, its pull-back to $Y^h$ is trivial too. The proposition is proved.
\end{proof}

\emph{Our aim is to modify the trivialization $\phi$ via an element
\[
   \alpha\in\bG(\dot Y^h)
\]
so that the $\bG$-bundle $\Gl(\cG',\phi\circ\alpha)$ becomes trivial over} $\P^1_U$.

\subsection{Principal bundles over open subsets of projective lines} We will recall some results from~\cite{GilleTorseurs}.
In this section $k$ denotes any field, $V$ denotes an open subscheme of $\P^1_k$, $G$ is a connected reductive group over $k$.

\begin{lemma}\label{lm:Gille}
\stepzero\noindstep\label{lm:Gille1}
A $G$-bundle over $V$ is locally trivial in the Zariski topology on $V$ if it is trivial at the generic point of $V$;

\noindstep\label{lm:Gille2}
Let $T$ be a maximal split torus of $G$, let $\hat T$ be its lattice of co-characters, and let $\Pic(V)$ denotes the group of isomorphism classes of line bundles over $V$. Then there is a natural surjection
\[
    \hat T\otimes_\Z\Pic(V)\to H^1_{\text{Zar}}(V,G).
\]
(Here $H^1_{\text{Zar}}$ stands for the set of isomorphism classes of Zariski locally trivial $G$-bundles.)
\end{lemma}
\begin{proof}
It is a reformulation of~\cite[Cor.~3.10(a)]{GilleTorseurs}, see also~\cite{GilleErratum}.
\end{proof}
Note that part~\eqref{lm:Gille1} of the lemma is a particular case of the Grothendieck--Serre conjecture. Note also, that the map in part~\eqref{lm:Gille2} is given as follows: given a co-character of $T$, we get a homomorphism $\mathbb G_{m,k}\to G$. Then every line bundle over $V$ yields a principal $G$-bundle via pushforward.

\subsection{Proof of Theorem~\ref{MainThm2}: proof of property~(i) from the outline}\label{sect:properties_i}
Now we are able to prove property (i) from the outline of the proof. In fact, we will prove the following
\begin{lemma}\label{podpravka}
Let $\Gl(\cG',\phi)$ be the presentation of the $\bG$-bundle $\cG$ over~$\P^1_U$ given in Proposition~\ref{presentation}. Set $\phi_\bu:=\phi|_{\dot Y_\bu^h}$. Then there is $\gamma_\bu\in\bG_\bu(\dot Y_\bu^h)$ such that the $\bG_\bu$-bundle $\Gl_\bu(\cG'|_{\P^1_\bu-Y_\bu},\phi_\bu\circ\gamma_\bu)$ is trivial.
\end{lemma}

\begin{proof}
We show first that $\cG|_{\P^1_\bu-Y_\bu}$ is trivial. Recall that $\bu'\subset\bu$ is the subscheme of all closed points $u_i$ such that the group $\bG_{u_i}$ is isotropic, and $\bu''=\bu-\bu'$. We can write
\[
    \P^1_\bu= \left(\coprod_{u\in\bu'}\P^1_u\right)\coprod\left(\coprod_{u\in\bu''}\P^1_u\right).
\]
For $u\in\bu$ set $Y_u:=Y\times_Uu$, $\bG_u:=\bG\times_Uu$, and $\cG_u:=\cG\times_Uu$.

For $u\in\bu''$ the algebraic $k(u)$-group $\bG_u$ is anisotropic. Since $\cG_u$ is trivial over an open subset of $\P^1_u$, Lemma~\ref{lm:Gille}\eqref{lm:Gille1} shows that $\cG_u$ is locally trivial in the Zariski topology. Now Lemma~\ref{lm:Gille}\eqref{lm:Gille2} shows that $\cG_u$ is trivial. Thus $\cG|_{\P^1_u-Y_u}$ is trivial.

Take $u\in\bu'$. By our assumption on $Y$, there is a $k(u)$-rational point $p_u\in Y_u$. Set $\A_u^1=\P_u^1-p_u$. Then we can write $Y_u=p_u\coprod T_u$ and $\P^1_u-Y_u\cong\A^1_u-T_u$. The $\bG_u$-bundle $\cG_u$ is trivial over $\A^1_u-Z$. Thus, again by Lemma~\ref{lm:Gille}, it is trivial over $\A^1_u$. Whence it is trivial over $\P^1_u-Y_u$.

We see that $\cG'|_{\P^1_\bu-Y_\bu}=\cG|_{\P^1_\bu-Y_\bu}$ is trivial. Choosing a trivialization, we may identify $\phi_\bu$ with an element of $\bG_\bu(\dot Y_\bu^h)$. Set $\gamma_\bu=\phi_\bu^{-1}$. By the very choice of~$\gamma_\bu$ the $\bG_\bu$-bundle $\Gl_\bu(\cG'|_{\P^1_\bu-Y_\bu},\phi_\bu\circ\gamma_\bu)$ is trivial.
\end{proof}

\subsection{Proof of Theorem~\ref{MainThm2}: reduction to property (ii) from the outline}\label{sect:properties_ii}
The aim of this section is to deduce Theorem~\ref{MainThm2} from the following
\begin{proposition}\label{alpha}
Each element $\gamma_\bu\in\bG_\bu(\dot Y_\bu^h)$ can be written in the form
\[
 \alpha|_{\dot Y_\bu^h}\cdot\beta_\bu|_{\dot Y_\bu^h}
\]
for certain elements $\alpha\in\bG(\dot Y^h)$ and $\beta_\bu\in\bG_\bu(Y_\bu^h)$.
\end{proposition}
\begin{proof}[Deduction of Theorem~\ref{MainThm2} from Proposition~\ref{alpha}]
Let $\Gl(\cG',\phi)$ be the presentation of the $\bG$-bundle $\cG$ from Proposition~\ref{presentation}. Let $\gamma_\bu\in\bG_\bu(\dot Y_\bu^h)$ be the element from Lemma~\ref{podpravka}. Let $\alpha\in\bG(\dot Y^h)$ and $\beta_\bu\in\bG_\bu(Y_\bu^h)$ be the elements from Proposition~\ref{alpha}. Set
\[
 \cG^{new}=\Gl(\cG',\phi\circ\alpha).
\]
\emph{Claim.} The $\bG$-bundle $\cG^{new}$ is trivial over $\P^1_U$.\\ Indeed, by Lemmas~\ref{basechange_limit} and~\ref{coboundary_limit} one has a chain of isomorphisms of $\bG_\bu$-bundles
\begin{multline*}
  \cG^{new}|_{\P^1_\bu}\cong
  \Gl_\bu(\cG'|_{\P^1_\bu-Y_\bu},\phi_\bu \circ \alpha|_{\dot Y_\bu^h})
\cong\\
\Gl_\bu(\cG'|_{\P^1_\bu-Y_\bu},\phi_\bu \circ
\alpha|_{\dot Y_\bu^h}\circ\beta_\bu|_{\dot Y_\bu^h})=
\Gl_\bu(\cG'|_{\P^1_\bu-Y_\bu}, \phi_\bu \circ \gamma_\bu);
\end{multline*}
the bundle $\Gl_\bu(\cG'|_{\P^1_\bu-Y_\bu}, \phi_\bu \circ \gamma_\bu)$ is trivial by the choice of $\gamma_\bu$. The $\bG$-bundles $\cG|_{\P^1_U-Y}$ and $\cG^{new}|_{\P^1_U-Y}$ coincide by the very construction of~$\cG^{new}$. By Proposition~\ref{pr:trivclsdfbr} applied to $T=Z\cup Y$ the $\bG$-bundle $\cG^{new}$ is trivial. Whence the claim.

The claim above implies that the $\bG$-bundle $\cG|_{\P^1_U-Y}$ is trivial. Theorem~\ref{MainThm2} is proved.
\end{proof}

\subsection{End of proof of Theorem~\ref{MainThm2}: proof of property (ii) from the outline}
\emph{In the remaining part of Section~\ref{sect:proof2} we will prove Proposition~\ref{alpha}. This will complete the proof of Theorem~\ref{MainThm2}}.

By assumption, the group scheme $\bG_Y=\bG\times_UY$ is isotropic. Thus we may choose a parabolic subgroup scheme $\bP^+$ in $\bG_Y$ such that the restriction of $\bP^+$ to each connected component of $Y$ is a proper subgroup scheme in the restriction of $\bG_Y$ to this component of $Y$.

Since $Y$ is an affine scheme, by~\cite[Exp.~XXVI, Cor.~2.3, Thm.~4.3.2(a)]{SGA3-3} there is an opposite to $\bP^+$ parabolic subgroup scheme $\bP^-$ in $\bG_Y$. Let $\bU^+$ be the unipotent radical of $\bP^+$, and let $\bU^-$ be the unipotent radical of $\bP^-$.

\begin{definition}\label{EYi}
If $T$ is a $Y$-scheme, we write $\bE(T)$ for the subgroup of $\bG_Y(T)=\bG(T)$ generated by the unipotent subgroups $\bU^+(T)$ and $\bU^-(T)$. Thus $\bE$ is a functor from the category of $Y$-schemes to the category of groups.
\end{definition}

\begin{lemma}\label{lm:surjectivity}
The functor $\bE$ has the property that for every closed subscheme $S$ in an affine $Y$-scheme $T$ the induced map $\bE(T)\to\bE(S)$ is surjective.
\end{lemma}
\begin{proof}
The restriction maps $\bU^\pm(T)\to\bU^\pm(S)$ are surjective, since $\bU^\pm$ are isomorphic to vector bundles as $Y$-schemes
(see~\cite[Exp.~XXVI, Cor.~2.5]{SGA3-3}).
\end{proof}

Recall that $(Y^h,\pi^h,s^h)$ is the Henselization of the pair $(\A_U^1,Y)$. Recall that $in:\A_U^1\to\P_U^1$ is the standard embedding. Denote the projection $\A_U^1\to U$ by $pr$ and the projection $\A_Y^1\to Y$ by $pr_Y$.
\begin{lemma}\label{lm:retraction}
There is a morphism $r:Y^h\to Y$ making the following diagram commutative
\begin{equation}\label{eq:retraction}
\begin{CD}
Y^h @>r>> Y\\
@V{in\circ\pi^h}VV @VV{pr|_{Y}}V\\
\P^1_U @>pr>> U
\end{CD}
\end{equation}
and such that $r\circ s^h=\Id_Y$.
\end{lemma}
\begin{proof}
As before, we may assume that $Y\subset\A_U^1$. Note that the morphism
\[
   \pi:=\Id\times(pr|_Y):\A^1_Y\to\A^1_U
\]
is \'etale. Let $s:Y\to\A^1_U\times_UY=\A^1_Y$ be the morphism induced by the embedding $Y\to\A^1_U$ and $\Id_Y$. Then $(\A_Y^1,\pi,s)\in\wN(\A_U^1,Y)$. Thus there is a canonical morphism $can:Y^h\to\A^1_Y$ such that $(\Id\times (pr|_Y))\circ can =\pi^h$. Set
\[
    r:= pr_Y\circ can: Y^h\to Y.
\]
With this $r$ diagram~\eqref{eq:retraction} commutes, and $r\circ s^h=\Id_Y$.
\end{proof}

We view $Y^h$ as a $Y$-scheme via $r$. Thus various subschemes of $Y^h$ also become $Y$-schemes. In particular, $\dot Y^h$ and $\dot Y_\bu^h$ are $Y$-schemes, and we can consider
\[
\bE(\dot Y^h)\subset\bG(\dot Y^h) \ \text{ and } \ \bE(\dot Y_\bu^h)\subset\bG(\dot Y_\bu^h)=\bG_\bu(\dot Y_\bu^h).
\]

\begin{lemma}\label{nastia}
\[
    \bG_\bu(\dot Y_\bu^h)=\bE(\dot Y_\bu^h)\bG_\bu(Y_\bu^h).
\]
\end{lemma}
\begin{proof}
Firstly, one has $Y_\bu=\coprod_{u\in\bu}\coprod_{y\in Y_\bu}y$. (Note that $Y_u$ is a finite scheme.) Thus by Lemma~\ref{lm:propii}, we have
\[
    Y_\bu^h=\coprod_{u\in\bu}\coprod_{y\in Y_u}y^h,\qquad
    \dot Y_\bu^h=\coprod_{u\in\bu}\coprod_{y\in Y_u}\dot y^h,
\]
where $(y^h,\pi^h_y,s^h_y)$ is the Henselization of the pair $(\A_\bu^1,y)$, $\dot y^h:=y^h-s_y^h(y)$. We see that $y^h$ and $\dot y^h$ are subschemes of $Y^h$, so we can view them as $Y$-schemes, and $\bG_{y^h}:=\bG_Y\times_Y{y^h}$ is isotropic. Also, $\bE(\dot y^h)$ makes sense as a subgroup of $\bG(\dot y^h)=\bG_u(\dot y^h)=\bG_{y^h}(\dot y^h)$.

There are equalities of the form
\[
\begin{split}
   \bG_\bu(\dot Y_\bu^h)&=\prod_{u\in\bu}\prod_{y\in Y_u}\bG_u(\dot y^h)=\prod_{u\in\bu}\prod_{y\in Y_u}\bG_{y^h}(\dot y^h),\\
   \bE(\dot Y_\bu^h)&=\prod_{u\in\bu}\prod_{y\in Y_u}\bE(\dot y^h),\\
   \bG_\bu(Y_\bu^h)&=\prod_{u\in\bu}\prod_{y\in Y_u}\bG_u(y^h)=\prod_{u\in\bu}\prod_{y\in Y_u}\bG_{y^h}(y^h).
\end{split}
\]
Thus, to prove the lemma it suffices for each $u\in\bu$ and each $y\in Y_u$ to check the equality
\[
   \bG_{y^h}(\dot y^h)=\bE(\dot y^h)\bG_{y^h}(y^h).
\]
Note that $y^h=\spec\mathcal O$, where $\mathcal O=k(u)[t]_{\mathfrak m_y}^h$ is a Henselian discrete valuation ring, and $\mathfrak m_y\subset k(u)[t]$ is the maximal ideal defining the point $y\in\A^1_u$. (Without loss of generality we can assume that $y$ is not the infinite point of $\P^1_u$.) Further, $\dot y^h=\spec L$, where $L$ is the fraction field of $\mathcal O$. Also, $\bG_{y^h}$ is isotropic. Thus, the equality follows from~\cite[Lemma~4.5(1)]{Gille:BourbakiTalk} in view of our definition of $\bE$ and~\cite[Fait~4.3(2)]{Gille:BourbakiTalk}.
\end{proof}

By Lemma~\ref{lm:surjectivity} and Proposition~\ref{pr:affine}\eqref{c} the restriction map $\bE(\dot Y^h)\to\bE(\dot Y_\bu^h)$ is surjective. Since $\bE(\dot Y^h)\subset\bG(\dot Y^h)$, the proposition follows. \emph{This completes the proof of Theorem~\ref{MainThm2}}.

\section{An application}\label{sect:application}
The following result is a straightforward consequence of Theorem~\ref{MainThm1} and an exact sequence for \'etale cohomology.
Recall that by our definition a reductive group scheme has geometrically connected fibres.
\begin{theorem}\label{Norms}
Let $R$ be a regular local ring containing an infinite field and $\bG$ be a reductive $R$-group scheme. Let $\mu:\bG\to\bT$ be a group scheme morphism to an $R$-torus $\bT$ such that $\mu$ is locally in the \'{e}tale topology on $\spec R$ surjective. Assume further that the $R$-group scheme $\bH:=\Ker(\mu)$ is reductive. Let $K$ be the fraction field of $R$. Then the group homomorphism
\[
\bT(R)/\mu(\bG(R))\to\bT(K)/\mu(\bG(K))
\]
is injective.
\end{theorem}
\begin{proof}
We have a commutative diagram whose rows are exact in the sense that in each row the image of $\mu$ coincides with the kernel of $\nu$.
\[
\begin{CD}
\bG(R) @>\mu>> \bT(R) @>\nu>> H^1_{\text{\'et}}(R,\bH) \\
@VVV @VVV @VVV\\
\bG(K) @>\mu>> \bT(K) @>\nu>> H^1_{\text{\'et}}(K,\bH).
\end{CD}
\]
By Theorem~\ref{MainThm1} the right vertical arrow has trivial kernel. Now a simple diagram chase completes the proof.
\end{proof}

This theorem extends all the known results of this form proved in~\cite{C-TO}, \cite{PS}, \cite{Z}, \cite{OPZ}. Theorem~\ref{Norms} has the following corollary.
\begin{corollary*}
Under the hypothesis of Theorem~\ref{Norms} let additionally the $K$-algebraic group $\bG_K$ be $K$-rational as a $K$-variety and let the ring $R$ be of characteristic $0$. Then the norm principle holds for all finite flat $R$-domains $S\supset R$.
That is, if $S\supset R$ is such a domain, and $a\in\bT(S)$ belongs to $\mu(\bG(S))$, then the element $N_{S/R}(a)\in\bT(R)$ belongs to $\mu(\bG(R))$.
\end{corollary*}
\begin{proof}
Let $L$ be the fraction field of $S$. Let $\alpha\in\bG(S)$ be such that $\mu(\alpha)=a\in\bT(S)$.
Then $\mu(\alpha_L)=a_L\in\bT(L)$, where $\alpha_L$ is the image of $\alpha$ in $\bG(L)$, $a_L$ is the image of $a$ in $\bT(L)$. The hypothesis on the algebraic $K$-group $\bG_K$ implies that there exists an element $\beta\in\bG(K)$ such that $\mu(\beta)=N_{L/K}(a_L)\in\bT(K)$ (see~\cite{Merk}). Note that $N_{L/K}(a_L)=(N_{S/R}(a))_K\in\bT(K)$. By Theorem~\ref{Norms} there exists an element $\gamma\in\bG(R)$ such that $\mu(\gamma)=N_{S/R}(a)\in\bT(R)$. Whence the corollary.
\end{proof}

\bibliographystyle{alphanum}
\bibliography{GrSerre}
\end{document}